\documentclass{amsart}
\usepackage[utf8]{inputenc}
\usepackage[foot]{amsaddr}
\usepackage{thm-restate}
\usepackage{tikz}
\usepackage{amsthm}
\usepackage{amsmath,amsfonts,amssymb}
\usepackage{thmtools, thm-restate}
\usepackage{parskip}
\usepackage{enumitem}

\renewcommand{\leq}{\leqslant}
\renewcommand{\geq}{\geqslant}

\newcommand{\ceil}[1]{\ensuremath{\lceil{} #1 \rceil{}}}

\newtheorem{theorem}{Theorem}[section]
\newtheorem{lemma}[theorem]{Lemma}

\newtheorem{clm}{Claim}

\theoremstyle{definition}

\newtheorem{definition}[theorem]{Definition}

\author{Matthew Wales}
\date{\today{}}
\address{DPMMS, University of Cambridge, CB3 0WB, UK}

\title{Bipartite clique minors in graphs of large Hadwiger number}

\begin{document}

\begin{abstract}
    The Hadwiger number $h(G)$ is the order of the largest complete minor in $G$. Does sufficient Hadwiger number imply a minor with additional properties?
    In \cite{GEELEN200920}, Geelen et al showed $h(G)\geq (1+o(1))ct\sqrt{\ln t}$ implies $G$ has a bipartite subgraph with Hadwiger number at least $t$, for some explicit $c\sim 1.276\dotsc$. We improve this to $h(G) \geq (1+o(1))t\sqrt{\log_2 t}$, and provide a construction showing this is tight. We also derive improved bounds for the topological minor variant of this problem.
\end{abstract}

\maketitle

\section{Introduction}

A well known result shows that every graph with average degree at least $2d$ contains a bipartite subgraph of average degree at least $d$, and further this result is essentially tight (for example, for complete graphs). Can we prove a similar result for other graph parameters? In a recent preprint, Hickingbotham and Wood \cite{HickWood} showed a number of such results, While some of their results were known, it serves as a helpful survey of the literature.

One particular choice of parameter is the Hadwiger number\footnote{A graph $H$ is a minor of $G$ if it can be obtained by a sequence of edge contractions, or vertex and edge deletions. The Hadwiger number $h(G)$ is the largest $t$ for which $K_t$ is a minor of $G$}. Geelen, Gerards, Reed, Seymour and Vetta \cite{GEELEN200920} showed that $h(G)\geq c_1t\sqrt{\ln t}$ implies $G$ contains a bipartite subgraph with Hadwiger number at least $t$ by combining results of Thomason \cite{ThomasonComplete} and a slight modification of the classical average degree result. The constant $c_1$ can be taken to be $1.276\dotsc + o(1)$, corresponding to twice the extremal function for $K_t$ minors.

Hickingbotham and Wood rederived this theorem, and explicitly posed the question of whether this bound is asymptotically tight. In this note, we show that the asymptotic growth rate is correct, however the constant is different. 

\begin{theorem}
\label{T:lb}
    For all $\epsilon > 0$, there is a $t_0$ such that for all $t>t_0$, if \\\mbox{$f(t) = \ceil{(1-\epsilon)t\sqrt{\log_2 t}}$}, there is a graph with a $K_{f(t)}$ minor, and in fact a $K_{f(t)}$ topological minor, which has no bipartite subgraph with Hadwiger number at least $t$.
\end{theorem}
We remark that our lower bound is asymptotically $(1.201\dotsc +o_t(1))t\sqrt{\ln t})$.

\begin{theorem}
\label{T:ub}
    For any $\epsilon > 0$ there is a $t_0(\epsilon)$ such that for any $t > t_0$, any graph with Hadwiger number at least $t\sqrt{\log_2 t}$ contains a bipartite subgraph with Hadwiger number at least $(1-2\epsilon)t$. 
\end{theorem}

Hickingbotham and Wood \cite{HickWood} also considered the natural analogue for topological minors\footnote{A graph $H$ is a topological minor of $G$ if $G$ contains as a subgraph a graph obtained from $H$ by replacing each edge with a path, all such paths being internally vertex disjoint}. If $tcl(G)$ is the order of the largest topological clique minor contained in $G$, combining the high average degree bipartite subgraph result and lower bounds on $tcl(G)$ in terms of average degree, they showed that there is a constant $c$ (in particular, one can take $c = 20/23 + o(1)$) such that for sufficiently large $t$, every graph with $tcl(G)\geq ct^2$ has a bipartite subgraph with $tcl(H)\geq t$. We were able to improve their upper bound, and also provide an example showing quadratic growth is necessary.

\begin{theorem}
    \label{T:topub}
    If $G$ is a graph with topological clique number at least $(\frac12 + o_t(1))t^2$, then $G$ has a bipartite subgraph $H$ with $tcl(H)\geq t$
\end{theorem}

\begin{theorem}
\label{T:toplb}
    There is a constant $C$ such that for all $t$, there is a graph $G$ with $tcl(G) \geq Ct^2$ but $G$ has no bipartite subgraph $H$ with $tcl(H)\geq t$.
\end{theorem}

Our proof takes $C =\frac14$, and the example is a complete graph. It seems reasonable to conjecture that $C = \frac14$ is tight (and indeed that cliques are extremal); our proof of $\frac12$ for an upper bound only requires this many vertices in a very structured case, which seems unlikely to occur. An existing upper bound of the shape $ct^2$ is twice the extremal function for topological minors. With the best known upper bound on this function (see \cite{KO_bestUB}), this gives a value $c = 20/23 + o(1)$. If it were possible to improve this upper bound to the best known lower bound, we would obtain $c = 9/32 + o(1)$ (this lower bound is from bipartite random graphs, and due to \L{}uczak) - there remains only a small gap between this and our lower bound. This suggests the lower bound is likely to be hard to improve, since a substantial improvement would also improve bounds on the extremal function.

\section{\mbox{RB-bipartite} graphs}
In this section, we introduce \mbox{RB-bipartite} graphs. These will serve as a general framework for proving the bounds. 

\begin{definition}
Let $H$ be a graph equipped with a 2-edge-colouring (with colours Red and Blue). Then $H$ is \textit{\mbox{RB-bipartite}} if every cycle uses an even number of red edges (call such cycles R-even). Equivalently, $H$ has no R-odd cycle.
\end{definition}

We note that this definition is not symmetric under interchanging colours: for instance, a $K_3$ coloured red is not \mbox{RB-bipartite}, but a blue coloured $K_3$ is. It is easily seen being \mbox{RB-bipartite} is equivalent to being bipartite if every blue edge is subdivided exactly once. The following lemma gives some other equivalent definitions.

\begin{lemma}
\label{T:equivdefsrb}
The following are equivalent.
\begin{enumerate}
    \item $H$ is $RB$-bipartite
    \item $H$ has no circuit using an odd number of red edges (no R-odd circuit)
    \item There is a partition $(X,Y)$ of $V(H)$ such that all edges between $X$ and $Y$ are red, and all edges within $X$ or within $Y$ are blue
\end{enumerate}
\end{lemma}

\begin{proof}
It is easily seen that $(3)\implies (2) \implies (1)$, and so it remains to prove that \mbox{RB-bipartite} graphs have a partition as in (3).
  Let $H_R$ denote the graph on the red edges, and $H_B$ the graph on blue edges. It is easily seen that $H_R$ is a bipartite graph. We will work on each component of $H_R$ in turn. Let $C_1$ be such a component, with bipartition $X_1,Y_1$. If there is a blue edge between $X_1$ and $Y_1$, then also taking a path in $H_R$ from between the endpoints we get a cycle using an odd number of red edges, contradiction. So the induced subgraph of $H$ on $C_1$ is \mbox{RB-bipartite}.
  
  Let $C_2$ be a different component of the red subgraph, with bipartition $X_2,Y_2$. Suppose there is a blue edge from $X_1$ to $X_2$, and also $X_1$ to $Y_2$. Then taking a path between the endpoints in $C_2$, as well as in $C_1$ (note this path may be empty), we get a cycle using an odd number of red edges. This contradicts that $H$ is \mbox{RB-bipartite}, and so we can (uniquely, if there is an edge between $C_1$ and $C_2$) extend the RB-bipartition to their union. 
  
  We consider each connected component of $H$ in turn, and note that such connected components are unions of red-connected components - suppose we have $C = \bigcup_{i=1}^r C_i$. Relabelling if necessary, we can assume that $H[\cup_{i\leq s}C_i]$ is connected for all $s$. Apply the above argument first to $C_1$ and $C_2$. Consider now $C_1\cup C_2$ (with the uniquely extended RB-bipartition) and $C_3$ =. The argument for components actually only used connectedness, and so we can apply it to this setting also to get a unique extension of the bipartitions. We now repeat with $\cup_{i\leq 3}C_i$ and $C_4$, and so on. This gives an RB-bipartition of $C$. We can therefore apply this in turn to each component, and the result follows.
\end{proof}

\begin{theorem}
\label{T:rbhalf}
    Let $H$ be a 2-coloured graph. Then $H$ has an \mbox{RB-bipartite} subgraph with at least $\frac12e(G)$ edges.
\end{theorem}

\begin{proof}
    Let $X,Y$ be a partition of $V(H)$. Define $e_R(X)$ to be the number of Red edges with both ends inside X, $e_R(X,Y)$ the number of red edges with one endpoint in $X$ and one endpoint in~$Y$.
    
    Suppose that we place vertices in $X$ or $Y$ independently at random with probability~$\frac12$. Let $d(X,Y) = e_R(X,Y) - e_B(X,Y)$. Then since each edge is between $X$ and $Y$ with probability~$\frac12$, we have $\mathbb{E}(d(X,Y)) = \frac12 (e_R(H) - e_B(H))$. Pick some choice of $X,Y$ that attains at least this expectation. Construct the subgraph $H'$ consisting of all blue edges inside $X$ or $Y$, and all red edges from $X$ to $Y$. Then by the lemma, $H'$ is \mbox{RB-bipartite}, and further \begin{align*}e(H') = (e_B(H) - e_B(X,Y)) + e_R(X,Y) \geq \frac12 (e_R(H) + e_B(H)) = \frac12 e(H)\end{align*}
    as desired.
\end{proof}

These graphs are introduced to simplify later proofs by relating minors to coloured graphs. For topological minors, if we take $G$ to be a subdivided $K_t$, and form a 2-coloured graph $H$ by replacing each path with a Red edge if it has odd length, and Blue if it has even length, a subgraph of $G$ consisting of a union of paths is bipartite if and only if the corresponding graph in $H$ is \mbox{RB-bipartite} (it is easily seen a cycle in $G$ has odd length if and only if the corresponding cycle of $H$ is R-odd). For general minors, things are more complicated, though in the following restricted setting we can obtain a result sufficient for our purposes.

\begin{definition}
\label{D:auxgraph}
Let $G$ be a graph, with a partition into $n$ parts $V_i$, such that all $G[V_i]$ are trees, and there is one edge from $V_i$ to $V_j$. Let $v_i\in V_i$ be an arbitrary choice of roots. The \textit{canonical path} from $v_i$ to $v_j$, where there is an edge from $V_i$ to $V_j$, is the unique $v_i v_j$ path in $G[V_i\cup V_j]$ 

The \textit{auxiliary graph} $H[G]$ has vertex set $[n]$, and an edge from $i$ to $j$ whenever there an edge between $V_i$ to $V_j$. This edge is coloured Red if the canonical $v_iv_j$ path has odd length, and Blue if it has even length.
\end{definition}

We remark that strictly this definition of $H[G]$ also depends on the choice of roots, however we suppress this notation as the choice does not matter.

\begin{lemma}
\label{T:bipiffrb}
Let $G, H[G]$ be as in Definition~\ref{D:auxgraph}, and let $H'$ be some (coloured) subgraph of $H$. Let $G'$ be a subgraph of $G$ containing all edges within a part $V_i$, together with the edge from $V_i$ to $V_j$ if and only if $ij\in E(H')$. Then $G'$ is bipartite if and only if $H'$ is \mbox{RB-bipartite}
\end{lemma}

\begin{proof}
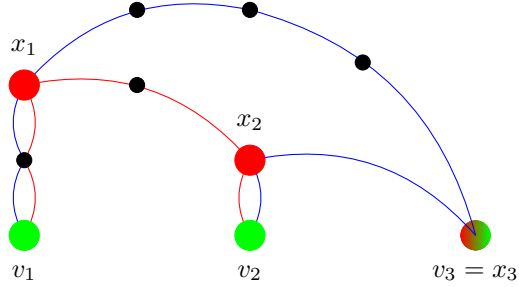
\begin{figure}
    \centering
\begin{tikzpicture}

    \shade[left color = red, right color = green] (6,0) circle (0.2cm);
    \draw[blue] (0,0) to[bend left] (0,1) to [bend left] (0,2);
    \draw[red] (0,0) to [bend right] (0,1) to [bend right] (0,2);
    \draw[red] (3,0) to[bend left] (3,1);

    \draw[blue] (3,0) to [bend right] (3,1);
    \draw[red] (0,2) to [bend left] (3,1);
    \draw[blue] (3,1) to [bend left] (6,0);
    \draw[blue] (0,2) to [bend left] (3,3) to [bend left] (6,0);
 
    \filldraw[green] (0,0) circle (0.2cm);
    \filldraw[green] (3,0) circle (0.2cm);
    \filldraw (0,1) circle (0.1cm);
    \filldraw[red] (0,2) circle (0.2cm);
    \filldraw[red] (3,1) circle (0.2cm);
    \filldraw (1.5,2) circle (0.1cm);
    \filldraw (1.5,3) circle (0.1cm);
    
    \filldraw (3,3) circle (0.1cm);
    \filldraw (4.5,2.3) circle (0.1cm);
      \node at (0,-0.5) {$v_1$};
    \node at (3,-0.5) {$v_2$};
    \node at (6,-0.5) {$v_3 = x_3$};
    \node at (0,2.5) {$x_1$};
    \node at (3,1.5) {$x_2$};
 
\end{tikzpicture}
    \caption{This figure illustrates how an odd cycle (or circuit) in the auxiliary graph necessarily produces an odd cycle in the main graph. The red path is the unique odd edge, the blue paths are even. The length of the sub-cycle of this odd circuit is also odd.}
    \label{fig:bipcycle}
\end{figure}

We start by proving the `only if' statement. By relabelling, we can assume we have a cycle $1,\dotsc,k$ in $H'$, where an odd number of edges $i,i+1$ are Red (and indices are taken modulo $k$). We aim to construct an odd cycle (or indeed circuit) in $G'$.

We build a circuit $C$ in $G$ as follows. Start by traversing the canonical path from $v_1$ to $v_2$ (and call this path $P_1$), then continue to $v_3$ along canonical path $P_2$, and so on. $C$ clearly has odd length, since it has an odd number of odd length segments, and so $G'$ cannot be bipartite (see Figure 1 for an illustration).

The forwards direction is similar; suppose that $G'$ is not bipartite, and consider an odd cycle $C$ in $G'$. Fix some starting point and orientation of the cycle. Relabelling, suppose that we pass through vertex classes $V_1,\dotsc,V_k$ (where we may have repetition) in this order - since there is only one edge between each pair of $V_i$, and each class is a tree, we have $k\geq 3$ and $V_{i-1},V_i,V_{i+1}$ are all distinct. Further, each $i$ is adjacent to $i+1$ in $H'$. Let $Q_i$ be the segment of $C$ inside $V_i$ (between $V_{i-1}$ and $V_{i+1}$, and inclusive of the edge to $V_{i-1}$). 

Since $V_i$ is connected, there is some path $P'_i$ inside $V_i$ from $v_i$ to this path segment which meets the path in exactly one vertex; call its endpoint $x_i$. Then following $P'_i$ from $v_i$ to $x_i$, then traversing $Q_i$ and $Q_{i+1}$ until we reach $x_{i+1}$, and finally traversing $P'_{i+1}$ to $v_{i+1}$ must be the canonical path from $v_i$ to $v_{i+1}$ (note that we never repeat a vertex, and hence have a path which is necessarily unique)

    This produces an odd circuit, since each new edge we add is traversed twice. But this circuit is a union of canonical paths, an odd number of which must have odd length. In particular, the corresponding circuit $1\dotsc k$ in $H'$ has an odd number of Red edges, and therefore $H'$ is not \mbox{RB-bipartite}.
\end{proof}

\section{Proof of theorem \ref{T:lb}}
\begin{proof}

Let $H$ be a graph on vertex set $[n]$. Consider the graph $G(H)$ consisting of $n$ special `branch vertices' $v_1,\dotsc,v_n$, and for each $i<j$ a path between $v_i$ and $v_j$ of length either 1 (i.e. an edge between them) if $ij$ is an edge of $H$, or 2 if not. Equivalently, start with a copy of $K_n$, and subdivide each edge $v_iv_j$ once when $ij$ is not an edge of $H$. We note the auxiliary graph $H[G]$ is a copy of $K_n$, with $H$ coloured Red and $H^c$ coloured blue.

Clearly, this graph $G(H)$ has Hadwiger number exactly $n$ for any $H$ (since we must always contract or delete all the internal vertices of induced paths for $t\geq 3$). We will let $H\sim G(n,\frac12)$ an Erd\H{o}s-Renyi random graph.

    What can we say about bipartite subgraphs of $G$? The subgraph of $G$ consisting of all paths where $ij$ is not an edge consists of some branch vertices, and some internal vertices. This is immediately seen to form a bipartite subgraph (or equivalently, this holds by Lemma ~\ref{T:bipiffrb})

    Applying this lemma, the maximal bipartite subgraphs of $G$ are formed by taking some bipartition $(X,X^c)$ of $[n]$, and taking all length 1 paths from $X$ to $X^c$, together with all length 2 paths contained in either $X$ or $X^c$.  We assume our graph is one of the at most $2^n$ graphs formed in this way.
    
    Fixing $X$, and sampling $H$ from $G(n,\frac12)$, each edge within or disjoint from $X$ has length 2 with probability $\frac12$, and each edge from $X$ to $X^c$ has length 1 with probability $\frac12$. This means that all maximal bipartite subgraphs are distributed like (subdivisions of) $G(n,\frac12)$. The subdivision does not affect the Hadwiger number. In particular, if we can show the probability $G(n,\frac12)$ has a $K_t$ minor is below $2^{-n}$, there is some choice of $H$ for which there is no bipartite $K_t$ minor.
    
    Bollob\'as, Catlin and Erd\H{o}s \cite{BollobasCatlinErdos} showed that the probability some (fixed) partition of $G = G(n,\frac12)$ into $s$ parts is `compatible' - i.e. there is an edge of $G$ between each pair of parts - is at most $\exp(-\binom{s}22^{-n^2/s^2})$. Note that the parts of a $K_s$ model\footnote{A collection $(V_h)_{h\in H}$ of subsets (called parts) of $V(G)$ is a model of $H$ in $G$ if each $V_h$ is connected, and when $h\sim h'$ in $H$, there is an edge of $G$ between $V_h$ and $V_{h'}$. $H$ is a minor of $G$ if and only if there is a $H$ model in $G$.} form a compatible partition. The number of choices of partition into $s<n$ parts is at most $n^n$. Let $s = n / \sqrt{\log_2 n - 3\log_2\log_2 n}$. 
    
    Then it is easily seen $2^n n^n \exp(-\binom{s}2 2^{-n^2/s^2}) = o(1)$, and in particular the probability $G(n,\frac12)$ has a $K_s$ minor is less than $2^{-n}$ for large $s$. This means there is some particular choice of $H$ for which no bipartite subgraph of $G(H)$ has Hadwiger number at least \mbox{$s = (1+o(1))n/\sqrt{\log_2n}\leq t$}.
\end{proof}

\section{Proof of Theorem~\ref{T:ub}}
\begin{proof}
Let $n = \ceil{t\sqrt{\log_2 t}}$, and let $G$ be a graph with $h(G)\geq n$. For the rest of this proof we will neglect rounding, since this can be absorbed into the value of $\epsilon$.

We delete as many edges as possible from $G$ while retaining a $K_{n}$ minor, and so we can assume that $G$ consists of $n$ disjoint subsets $V_1,\dotsc,V_n$, with each $G[V_i]$ a tree, as well as exactly one edge between each pair $V_i$ and $V_j$ ($i\neq j$).

Recall from Definition~\ref{D:auxgraph} the canonical path and auxiliary graph $H[G]$; we will use these notions here. We start by reserving  a set $S\subset V(H)$ of $\epsilon n$ vertices for later use, and apply our theorems with $G' = G[\cup_{i\notin S}V_i]$. Note that $H[G']$ is a 2-coloured $K_{(1-\epsilon)n}$.

By Theorem~\ref{T:rbhalf}, $H[G']$ has an \mbox{RB-bipartite} subgraph $H'$ (with \mbox{RB-bipartition} $(A,B)$ ) on at least $\frac12 \binom{(1-\epsilon)n}{2}$ edges. By Theorem~\ref{T:bipiffrb}, the corresponding subgraph $G''$ of $G'$ (consisting of all edges inside $G'[V_i]$, as well as the edge from $V_i$ to $V_j$ whenever $ij$ is an edge of $H'$) is bipartite, and this also holds for any proper subgraph of $H'$. Note also that a minor in $H'$ directly corresponds to a minor in $G''$ (by first contracting each $V_i$ to a vertex), and hence to a minor in $G'$ which is bipartite.

\begin{definition}

Given a graph $F$, a collection of subsets $(U_v)_{v\in V(K_m)}$ is called a \textit{$K_m$-compatible-partition} if it forms a $K_m$ minor without the connectedness condition, i.e. the subsets are pairwise disjoint, and for each pair $vw$, there is an edge in $F$ between $U_v$ and $U_w$.
\end{definition}

The following result of Thomason \cite{ThomasonComplete} shows that $H'$ contains a $K_m$-compatible-partition for $m = (1- 3\epsilon/2))n/\sqrt{\log_2 n} = (1-3\epsilon/2 )t$, provided $t$ (and hence $n$) is large by taking a subgraph of $H'$ of density between $\frac12$ and $\frac34$.

\begin{theorem}[\cite{ThomasonComplete}]
    Let $\epsilon > 0$. Then there is an $n_0(\epsilon)$ such that for all \\\mbox{$(\log \log n)^{2+\epsilon} < p <1 - (\log n)^{-1/\epsilon}$}, and all graphs $F$ with $n>n_0$ vertices and density $p$, $F$ has a $K_m$-compatible-partition for some $m\geq (1-\epsilon)n/\sqrt{\log_{1/1-p}n}$
\end{theorem}

If in addition $H'$ had `reasonably high' connectivity, further results of that paper \cite{ThomasonComplete} would now imply that we can extend this to a minor, with the proof implicitly using (in the language of \cite{ThomasonWales}) two subsets $P$ (the projector) to project subsets to logarithmically smaller sets, and a $C$ (the connector) to connect small sets. However, we do not know that $H'$ has suitable connectivity, and so instead use the reserved set $S$ to serve these roles.

Suppose our $K_m$ compatible partition in $H'$ consists of the pairwise adjacent disjoint subsets $U_1,\dotsc,U_m$ of $V(H')$.

We start by performing the role of the projector. Let us first restrict attention to $U_1$, and suppose we have $U_A = U_1 \cap A$, $U_B = U_1\cap B$ (recalling $(A,B)$ is some fixed \mbox{RB-bipartition} of $H'$). Let $i$ be some element of $S\setminus V(H')$. We want to add $i$ to the subgraph $H'$ in such a way we preserve \mbox{RB-bipartiteness}. Adding $i$ to $A$, we can add all Red edges from $i$ to $B$, and all Blue edges from $i$ to $A$; adding $i$ instead to $B$ reverses the colours. 
In particular, one of these choices must allow us to add edges from $i$ to at least half of $U$. Make such a choice, and add $i$ and all such edges to $H'$. Replace $U$ with the set of vertices in $U$ not adjacent in $H'$ to $i$, and iterate this process. In this way, we can construct a set $P(U_1)$ (and add its vertex set to $H'$) of size at most $p(U_1) = (1 + \log_2 |U_1|)$, add some extra edges to $H'$ such that every element of $U_1$ has a neighbour in $P(U_1)$, and extend the \mbox{RB-bipartition} $(A,B)$ of $H'$ to these new vertices and edges. 

We can construct such sets in turn for each of $U_1,\dotsc,U_m$ using at most \\\mbox{$\sum p(U_i) \leq t + t\log(n/t)$} elements of $S$ overall, assuming $|S| > 2t + t\log(n/t)$. Since this holds for our value of $n$, we can obtain an \mbox{RB-bipartite} extension $H''$  of $H'$, which also contains disjoint new sets $P(U_i)$ of size at most $p(U_i)$, where each vertex of $U_i$ has a neighbour (in $H'$) inside $P(U_i)$. Let $S' = S \setminus \cup_i P(U_i)$, and note $|S'|\geq \epsilon n / 2$.

\begin{lemma}
Let $H'$ be an \mbox{RB-bipartite} graph, $x,y\in V(H')$, and let $T\subset S'$ have size $t$, where $S'$ is disjoint from $H'$. Suppose that every vertex of $T$ is joined to every vertex of $T\cup H'$ by either a red or a blue edge.
Then either $T$ forms an \mbox{RB-bipartite} $K_t$, or there is some set of at most 2 vertices from $T$ which can be added to $H'$ to form the internal vertices of an $xy$ path while preserving \mbox{RB-bipartiteness}.
\end{lemma}
\begin{proof}
Let $H'$ have \mbox{RB-bipartition} (A,B).
Suppose first that $x,y\in A$ (this corresponds to joining $v_i$ and $v_j$ by a path using an even number of Red edges). If some vertex is joined to both $x$ and $y$ by edges of the same colour, we can create a length 2 path from $x$ to $y$, adding the new vertex to $A$ or $B$ depending on parity.

So we can assume this does not happen. Let $X$ be the set of vertices joined to $x$ by a red edge, and $Y$ those vertices joined to $y$ by a red edge; note this forms a partition of $T$. If some pair of vertices $v,w$ in $X$ are joined by a red edge, then $xvwy$ is an even length path from $x$ to $y$; we can extend our bipartition by adding $v$ to $A$ and $w$ to $B$. The same argument works if $v,w\in Y$. 

If $v\in X, w\in Y$, and $vw$ is blue, then we likewise get a length three path. So we can assume that these cases never happen. But then $(X,Y)$ forms an \mbox{RB-bipartition} of $T$. But since the graph on $T$ is complete, we directly have an \mbox{RB-bipartite} $K_t$ subgraph.

 If instead $x\in A, y\in B$ the argument is similar. If some vertex is joined to $x$ and $y$ by different colours, we obtain a path of length 2 from $x$ to $y$ which is R-odd. Therefore, we can assume the set $R$ of vertices joined to $x$ and $y$ by red edges, and $B$ for blue edges form a partition of $T$. As above, and illustrated in Figure 2, if there is a red edge within either $R$ or $B$, or a blue edge between $R$ and $B$, we get an R-odd path from $x$ to $y$ of length at most 3. So we can assume this doesn't happen, and therefore the partition $(R,B)$ shows that $T$ is directly an \mbox{RB-bipartite} $K_t$.
 \end{proof}
 
This shows in particular that if $|S'|\geq t + 2l$, we can sequentially connect any $l$ pairs in a disjoint fashion while preserving \mbox{RB-bipartiteness}.

We will now use Lemma 4.2 and the set $S'$ to sequentially connect pairs from $H''$ and hence obtain a minor; we can assume that we never obtain an \mbox{RB-bipartite} $K_t$ subgraph in the lemma as otherwise we are directly done. We will first label $P(U_1)$ as $x_1,\dotsc,x_{|P(U_1)|}$. 

Since $\sum p(U_i) + t \leq |S'|/2$, we can use two new elements of $S$ to join $x_1$ to $x_2$, then $x_2$ to $x_3$ and so on while remaining \mbox{RB-bipartite}. We can then repeat this for $U_2$ and so on. Suppose that the internal vertices of the paths used to connect $P(U_i)$ are $C(U_i)$. We now replace $U_i$ with $V_i = U_i \cup P(U_i) \cup C(U_i)$, and extend the subgraph $H''$ [in a \mbox{RB-bipartite} fashion] to include $\cup C(U_i)$, with the edges along the paths also added. $V_i$ forms a connected subset containing $U_i$, and so the $V_i$ form an \mbox{RB-bipartite} $K_m$ minor. This corresponds to a bipartite $K_m$ minor in $G$.
\end{proof}

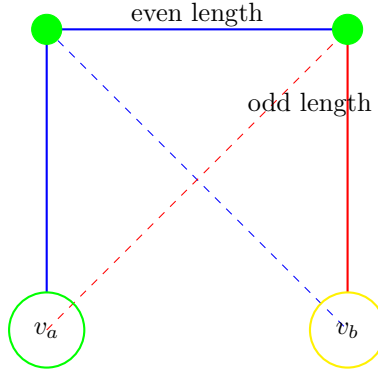
\begin{figure}
    \centering
    \begin{tikzpicture}

\draw[blue,thick] (0,0.5) -- (0,4);
\draw[red,thick] (4,0.5) -- (4,4);
\draw[blue,thick] (0,4) -- (4,4);
\draw[blue,dashed] (0,4) -- (4,0);
\draw[red,dashed] (0,0) -- (4,4);

\draw[green,thick] (0,0) circle (0.5cm);
\node at (0,0) {$v_a$};
\draw[yellow,thick] (4,0) circle (0.5cm);
\node at (4,0) {$v_b$};
\filldraw[green] (0,4) circle (0.2cm);
\filldraw[green] (4,4) circle (0.2cm);
\node at (2,4.2) {even length};
\node at (3.5,3) {odd length};

\end{tikzpicture}
    \caption{The green and yellow nodes represent $A$ and $B$ respectively. The thick lines show how to add the edge while remaining bipartite.}
    \label{fig:bipaddedge}
\end{figure}
\section{Results for topological minors}

\begin{proof}[Proof of Theorem~\ref{T:topub}]

We reduce the theorem to the following claim.
This corresponds to topological minors in $G$ where the branch vertices are preserved
    \begin{clm}
    Let $G$ be a 2-coloured copy of $K_{\binom{t+3}{2}} = K_{t^2/2 + O(t)}$. Then $H$ contains an \mbox{RB-bipartite} topological $K_t$ minor.
    \end{clm}
    
    We will build up a bipartite $K_s$ model one vertex at a time; the base case \\\mbox{$s=1$} being trivial. We impose the additional constraint that our $TK_s$ (topological $K_s$) uses at most $1 + \binom{s+1}{2}$ vertices to help with our inductive step.
    
    Suppose that $V = \{v_1,\dotsc,v_s\}$ are the branch vertices of our $TK_s$, and that $S$ is the set of vertices not appearing in the $TK_s$. Pick some vertex $w$ of $S$ at random. Let $(A,B)$ be a bipartition of the vertices of the $TK_s$. 
By a result used in the proof for minors, adding $w$ to either $A$ or $B$ allows us to retain at least half of the edges from $w$; pick some choice where we can do so. It remains to build subdivided edges from $w$ to the remaining at most $s/2$ vertices.

Pick some vertex $v_i$ we do not yet have a path to. Suppose that we are trying to build a R-odd path to $v_i$ from $w$. Let $S_1$ be those vertices in $S$ red-joined to $v_i$, and $S_2$ those vertices which are blue-joined.

If some vertex of $S_1$ is blue-joined to $v_i$, we can take this path of length 2 and extend our model. The same argument works if a vertex of $S_2$ is blue-joined to $v_i$. So we can assume this never happens.

If $S_1$ has size at least $t$, but contains no red edges, we directly have a blue (and hence $\mbox{RB-bipartite}$) $TK_t$. So we can assume there is some red edge $xy$ in $S_1$ in this case. But then the path $wxyv_1$ has 3 red edges, and so is R-odd.

If instead $S_2$ has size at least $t$, we can also use two internal vertices to extend our minor, or otherwise directly obtain a bipartite $TK_t$. If we want to join by an R-even path, we swap the colours of all edges incident to $w$, apply the above argument, and then swap back. 
In particular, we can use at most $1+ 2(s/2) = s+1$ additional vertices to turn our bipartite $K_s$ topological minor into an \mbox{RB-bipartite} $TK_{s+1}$, provided at least $2t+s+1$ vertices remain. Since the $TK_s$ had at most $1+ \binom{s}{2}$ vertices, this new topological minor has at most $1+ \binom{s+1}{2}$ vertices. In particular, taking $n > 2t + 1 +\binom{t+1}{2}$, $H$ has an \mbox{RB-bipartite} $TK_t$.

\end{proof}

\begin{proof}[Proof of Theorem ~\ref{T:toplb}]
Our lower bound relies on complete bipartite graphs being good examples of graphs with high average degree, and no topological $K_t$ minor.

Let $G$ be a complete bipartite graph on parts $A,B$, and suppose that there is a topological $K_t$ minor, say with $s$ branch vertices in $A$ and $(t-s)$ branch vertices in $B$. Then for each pair of vertices (in $A$, say), we need a distinct vertex in $B$ to lie on the path between them. This means we need $|B|\geq (t-s) + \binom{s}{2}$, and so overall $|G| \geq t + \binom{s}{2} + \binom{t-s}{2} \geq t + 2\binom{t/2}{2} = t^2/4 + t/2$ by convexity.

If we take $G$ to be a complete graph $K_{\ceil{t^2/4}}$, no bipartite subgraph can have a topological $K_t$ minor (as there are too few vertices). The result follows as $tcl(K_n) = n$.

\end{proof}

\subsection*{Acknowledgements}
The author was supported by an EPSRC DTP Studentship. The author would also like to thank Andrew Thomason for some helpful discussions.

\bibliography{biblio.bib}
\bibliographystyle{plain}

\end{document}